\documentclass[12pt,reqno]{amsart}

\usepackage{amssymb}

\newtheorem{theorem}{Theorem}[section]
\newtheorem{proposition}[theorem]{Proposition}
\newtheorem{lemma}[theorem]{Lemma}
\newtheorem{corollary}[theorem]{Corollary}

\theoremstyle{definition}
\newtheorem{definition}[theorem]{Definition}

\numberwithin{equation}{section}

\begin{document}

\baselineskip=15pt

\title{Generalized holomorphic Cartan geometries}

\author[I. Biswas]{Indranil Biswas}

\address{School of Mathematics, Tata Institute of Fundamental
Research, Homi Bhabha Road, Mumbai 400005, India}

\email{indranil@math.tifr.res.in}

\author[S. Dumitrescu]{Sorin Dumitrescu}

\address{Universit\'e C\^ote d'Azur, CNRS, LJAD, France}

\email{dumitres@unice.fr}



\dedicatory{Dedicated to the memory of Stefan Papadima}

\date{}

\begin{abstract} 
This is largely a survey paper, dealing with Cartan geometries in the complex analytic category. We 
first remind some standard facts going back to the seminal works of F. Klein, E. Cartan and C. 
Ehresmann. Then we present the concept of a {\it branched holomorphic Cartan geometry} which was 
introduced by the authors in \cite{BD}. It generalizes to higher dimension the notion of a 
branched (flat) complex projective structure on a Riemann surface introduced by Mandelbaum.
This new framework is much more flexible than that of the usual holomorphic Cartan geometries 
(e.g. all compact complex projective manifolds admit branched holomorphic projective 
structures). At the same time, this new definition is rigid enough  to enable us to classify 
branched holomorphic Cartan geometries on compact simply connected Calabi-Yau manifolds.
\end{abstract}

\maketitle

\tableofcontents

\section{Introduction}

An important consequence of the uniformization theorem for Riemann surfaces asserts, in particular, that any Riemann surface $M$ 
admits a holomorphic atlas with coordinates in ${\mathbb C}P^1$ and transition maps in $\text{PSL}(2, {\mathbb C})$. In other words, 
any Riemann surface is locally modelled on the complex projective line. This defines a {\it (flat) complex projective structure} on 
$M$. Complex projective structure on Riemann surfaces were introduced in connection with the study of the second order ordinary 
differential equations on complex domains and had a very major role to play in understanding the framework of the uniformization 
theorem \cite{Gu,StG}.

The complex projective line ${\mathbb C}P^1$ acted on by the M\"obius group $\text{PSL}(2, {\mathbb C})$ is a geometry in the 
sense of Klein's Erlangen program in which he proposed to study all geometries in the unifying frame of the homogeneous model spaces $G/H$, 
where $G$ is a finite dimensional Lie group and $H$ a closed subgroup of $G$.

Thus homogeneous spaces could be used as good models to geometrize higher dimensional topological 
manifolds. Indeed, following Ehresmann \cite{Eh}, a manifold $M$ is locally modelled on a 
homogeneous space $G/H$, if $M$ admits an atlas with charts in $G/H$ satisfying the
condition that the transition maps are given 
by elements of $G$ using the left-translation action of $G$ on $G/H$. In this way $M$ is locally 
endowed with the $G/H$-geometry and all $G$-invariant geometrical features of $G/H$ have 
intrinsic meaning on $M$.

Elie Cartan generalized Klein's homogeneous model spaces to {\it Cartan geometries} 
(or {\it Cartan connections}) (see Definition \ref{del} in Section \ref{s2.1}). We recall 
that those are geometrical structures infinitesimally modelled on homogeneous 
spaces $G/H$. A Cartan geometry on a manifold $M$ is equipped with a curvature tensor 
(see Definition \ref{del}) which vanishes exactly when $M$ is locally 
modelled on $G/H$ in the sense of Ehresmann \cite{Eh}. In such a situation the 
Cartan geometry is called {\it flat}.

In \cite{M1, M2} Mandelbaum introduced and studied {\it branched affine and projective structures} on Riemann 
surfaces. Those branched projective structures on Riemann 
surfaces are given by some holomorphic atlas where local charts are finite branched 
coverings on open sets in ${\mathbb C}P^1$ while the transition maps lie in 
$\text{PSL}(2, {\mathbb C})$. 

Inspired by the above mentioned works of Mandelbaum, we defined in \cite{BD} a more general 
notion of {\it branched holomorphic Cartan geometry} on a complex manifold $M$ (see Section 
\ref{s2.2}), which is valid also in higher dimension for non necessarily flat Cartan geometries. 
We define and study here a more general notion of {\it generalized holomorphic Cartan geometry} (see 
Definition \ref{general} and compare with \cite{AM}).

This new notions of generalized (and branched) Cartan geometry are much more flexible than the usual one. For example, all compact 
complex projective manifolds admit a branched flat holomorphic projective structure (see Proposition \ref{algebraic proj struct}). This 
will be obtain as a consequence of the fact that branched Cartan geometries are stable by pull-back through holomorphic ramified maps. 
Moreover, generalized Cartan geometries are stable by pull-back through any holomorphic map.

At the same time this notion is rigid enough to enable one to obtain classification results. We 
prove here the following results (see Corollary \ref{flatness} and 
compare with \cite{BD} for branched Cartan geometry):

{\it Any generalized holomorphic Cartan geometry of type $(G,H)$, with $G$ a complex semi-simple Lie 
group, on a compact simply connected K\"ahler Calabi--Yau manifold is flat.}

This yields to the following corollary:

 {\it Non-projective compact simply connected 
K\"ahler Calabi--Yau manifolds do not admit any branched holomorphic projective 
structure.}

The previous results are consequences of the following Theorem \ref{trivial} ([Theorem 6.2, \cite{BD}]) which is of independent interest:

{\it If $E$ is a holomorphic 
vector bundle over a compact simply connected K\"ahler Calabi--Yau manifold, and $E$ 
admits a holomorphic connection, then $E$ is a trivial holomorphic vector bundle 
equipped with the trivial connection.}

This statement gives a partial answer of a more general (still open) question asked by Atiyah \cite{At}: {\it Let $E$ be a holomorphic 
vector bundle over a compact complex K\"ahler manifold. Assume that $E$ admits a holomorphic connection. Is it true that $E$ also 
admits a flat connection ?}

The question is natural since, under the previous hypothesis, the vector bundle $E$ has trivial rational Chern classes \cite{At}
and hence there are no topological obstructions to admit a flat connection.

An original result of this paper is Proposition \ref{propc} which constructs nontrivial principal elliptic bundles over K3 surfaces bearing holomorphic connections. In particular, those elliptic bundles do not admit flat connections and give a counter-example to
Atiyah's (generalized) question to the context of principal (elliptic) bundles over (simply connected) K\"ahler manifolds.

\section{Klein geometries and $(G,X)$-structures}

Following F. Klein's point of view in his Erlangen's address (1872), a {\it geometry} is a space 
(manifold) $X$ endowed with a transitive action of a (finite dimensional) connected Lie group 
$G$. In this context, two subsets $P_1$ and $P_2$ in $X$, will be considered {\it congruent} 
(also called equivalent) if there exists an element $g \in G$ such that $g \cdot P_1=P_2$.

Let us (arbitrarily) pick-up a point $x_0 \in X$. Then $X$ is diffeomorphic to a quotient $G/H$, where $H$ is the closed subgroup of $G$ stabilizing $x_0$. 

Of course, the first example to think of is that of the $n$-dimensional Euclidean geometry for 
which $G \,=\, \text{O}(n, \mathbb{R}) \ltimes {\mathbb R}^n$ is the 
group of Euclidean motions (it is the semi-direct product of the corresponding orthogonal group $ \text{O}(n, {\mathbb R})$ with 
the translation group for the standard action
of the orthogonal group) and $H \,=\, \text{O}(n, {\mathbb R})$ is the stabilizer of the origin.

In the complex analytic category, it will be suitable to consider instead the {\it complex 
Euclidean space} $(X \,=\,{\mathbb C}^n;\, dz_1^2 + \ldots + dz_n^2)$ for which $G$
is the semi-direct product $\text{O}(n, 
\mathbb{C}) \ltimes {\mathbb C}^n$ for the standard action of
$\text{O}(n, \mathbb{C})$ on ${\mathbb C}^n$ (it is the group of complex 
Euclidean motions) and $H\,=\,\text{O}(n , \mathbb{C})$ (the stabilizer of the origin). The 
non-degenerate complex quadratic form $dz_1^2 + \ldots + dz_n^2$ is a flat holomorphic 
Riemannian metric (see \cite{Gh,D,DZ}).

One also studies classically the complex affine space $X\,=\,{\mathbb C}^n$ under the action of the 
complex affine group $G\,=\,\text{GL}(n,\mathbb C) \ltimes {\mathbb C}^n$
(the semi-direct product for the standard action of $\text{GL}(n,\mathbb C)$ on ${\mathbb C}^n$)
which sends lines 
parametrized at constant speed on (possibly other) lines parametrized at constant speed and the 
complex projective space. Recall that the complex projective space ${\mathbb C}P^n$ acted on by 
the complex projective group $\text{PGL}(n+1, \mathbb C)$ is a geometry where the symmetry group 
preserves (complex) collinearity.
 
In a program which is a vast generalization of Klein's geometries, E. Cartan created the framework of {\it 
Cartan geometries} which are infinitesimal versions of those $G/H$-homogeneous spaces (exactly 
in the same way as Riemannian geometry is the infinitesimal version of the Euclidean space).
Those Cartan geometries which will be presented in Section \ref{s2} admit a curvature tensor 
which will be zero exactly when the Cartan geometry is locally isomorphic to a Klein geometry. 
The Cartan geometries for which the curvature tensor vanishes
are called {\it flat}, and they were nicely described from the global 
geometric point of view by C. Ehresmann in \cite{Eh}.

Ehresmann's definition of those so-called {\it locally homogeneous spaces} (presently the 
terminology {\it $(G,X)$-structure} is more in use) is the following. Let $X$ be a $G$--homogeneous space.
 
\begin{definition} A $(G,X)$-structure on a manifold $M$ is given by an open cover $(U_i)_{i \in I}$ of $M$ with charts
$\phi_i \,:\, U_i \,\longrightarrow\, X$ such that the transition maps
$$\phi_i \circ \phi_j^{-1} \,:\, \phi_j(U_i \cap U_j)\,\longrightarrow\, \phi_i(U_i \cap U_j)$$
are given (on each connected component) by the restriction of an element in $G$.
\end{definition}
 
Two $(G,X)$-structures on $M$ will be identified if there exists a diffeomorphism of $M$ 
isotopic to the identity map that sends the first $(G,X)$-structure to the second one.
 
The key point of the previous definition is the fact that any $G$-invariant geometric feature of 
$X$ will have an intrinsic meaning on $M$.
 
Ehresmann's results in \cite{Eh} led to the following geometric description of this situation:
 
\begin{theorem}[Ehresmann]\label{Ehresmann}
 Let $X\,=\,G/H$ be a Klein geometry and $M$ a manifold of the same dimension as $X$. Then there
is a canonical one-to-one correspondence between the following structures:
\begin{enumerate}
\item A ($G,X)$-structure on $M$.

\item A fiber bundle $E_X$ over $M$ with fiber $X$ and structure group $G$, equipped with
a flat connection together with global section which is transversal to the horizontal
distribution for the connection.

\item A principal $H$-bundle $E_H$ over $M$ and a flat connection on the associated
principal $G$-bundle $E_G\,=\,E_H \times^H G$, such that the subbundle $E_H$ is transverse to
the horizontal distribution for the connection.
\end{enumerate}
\end{theorem}

Let us briefly sketch a proof of Ehresmann's theorem. 

\begin{proof}
We will first prove that (1) implies (2). Consider a cover of $M$ by open subsets $U_i$ and local charts $\phi_i \,:\, U_i 
\,\longrightarrow\, X$ with the property that on intersections $U_i \cap U_j$ there exist elements $g_{ij} \in G$ such that $\phi_i \circ 
\phi_j^{-1}\,=\,g_{ij}$. Then $g_{ij}$ is a locally constant cocycle with values in $G$ and hence it defines a flat
principal $G$-bundle $E_G$ over 
$M$. Consider the associated flat bundle $E_X\,=\, E_G/H$, with fiber $X\,=\,G/H$. Then the charts $\phi_i$ are local sections of
$E_X$ which glue together compatibly to produce
a global section of $E_X$. The flat connection of $E_X$ is locally determined by the foliation $U_i \times \{x\}$, for $x 
\,\in\, X$. Since the local charts $\phi_i$ are local diffeomorphisms, the associated section is transverse to this foliation and hence to 
the corresponding flat connection. It may be mentioned that, since $X\,=\,G/H$, a global section of $E_X$ produces
a reduction of structure group of the principal $G$-bundle $E_G$ to the subgroup
$H\, \subset\, G$. Indeed, for any section $\phi'\, :\, M\, \longrightarrow\, E_X$ of the projection
$E_X\,=\, E_G/H\, \longrightarrow\, M$, the inverse image of $\phi'(M)\, \subset\, E_X$ under the quotient map
$E_G\, \longrightarrow\, E_G/H$ is principal $H$--bundle on $M$. Therefore, the section of $E_X\, \longrightarrow\, M$
given by the charts $\phi_i$ produces a  reduction $$E_H\, \subset\, E_G$$ of structure group of the principal
$G$-bundle $E_G$ to the subgroup $H$. Also since the 
previous section of $E_X$ is transverse to the flat structure of $E_X$, we have that $E_H$ is transverse to the flat structure of 
$E_G$. Hence we proved that (1) also implies (3). The same proof also shows that (2) implies (3).

Let us finish the proof by showing that (3) implies (1). Consider a cover of $M$ by simply connected open sets $U_i$. Since the
connection on $E_G$ is flat, over each $U_i$, the principal
$G$-bundle $E_G$ equipped with the flat connection must be isomorphic to the trivial principal
$G$-bundle $U_i \times G\, \longrightarrow\, U_i$ equipped with the trivial connection. Therefore, over each $U_i$, the
reduction of structure group $E_H\vert_{U_i}\, \subset\, E_G\vert_{U_i}$
is defined by a smooth map $$\phi_i \,:\,U_i \,\longrightarrow\, X\,=\,G/H\, .$$
Now the transversality condition implies that $\phi_i$ 
is a local diffeomorphism. On $U_i \cap U_j$ the two maps $\phi_i$ and $\phi_j$ agree up to multiplication by an element $g_{ij} \,\in\, 
G$, because any two local trivializations of a principal
$G$-bundle $E_G$ equipped with a flat connection differ by the left-multiplication (on the trivial
principal $G$-bundle) by an element of $G$. This completes the proof.
\end{proof}

Let $\widetilde{M}\, \longrightarrow\, M$ be a universal cover.
Let us pull-back to $\widetilde{M}$ the flat bundle $E_X$ in description (2) in Theorem \ref{Ehresmann}.
This pull-back is isomorphic to the trivial fiber bundle $\widetilde{M} \times G/H\, \longrightarrow\,
\widetilde{M}$ equipped with the trivial connection. Fix an isomorphism between these two flat fiber bundles.
The section of the fiber bundle $\widetilde{M} \times G/H\, \longrightarrow\, \widetilde{M}$, given by the
pull-back of the section of $E_X\, \longrightarrow\, M$ using the chosen isomorphism of flat bundles,
is called the developing map $dev$ of the $(G,X)$-structure. It is a map from 
$\widetilde{M}$ to $X$, which is unique up to post-composition with an element in $G$, because we had to choose
a trivialization of the flat fiber bundle over $\widetilde{M}$. This developing map is a local 
diffeomorphism because of the transversality condition.

Let $\rho \,:\, \pi_1(M) \,\longrightarrow\, G$ be the monodromy representation for the flat connection on $E_X$
over $M$; since the structure group of $E_X$ is $G$, the monodromy representation is a homomorphism to $G$.
The construction of $\rho$ also involves choosing a trivialization of the pullback of the flat bundle $E_X$ to the universal
cover $\widetilde M$. So the homomorphism $\rho$ is unique up to inner automorphism by an
element of $G$. The fiber bundle $E_X$ is a 
quotient of $\widetilde{M} \times (G/H)$ through the identification $$(p,x)\,=\,(\gamma \cdot p,\, \rho 
(\gamma) \cdot x)\, ,$$ with $\gamma \,\in\, \pi_1(M)$ acting on $\widetilde{M}$ by deck-transformation and 
on $X$ through $\rho$. Since developing map is the pull-back of a section of $E_X$, it is equivariant 
with respect to the monodromy morphism $dev(\gamma \cdot p)\,=\,\rho(\gamma) \cdot dev(p)$.
 
In the description (3) in Theorem \ref{Ehresmann}, one can obtain the developing map and the monodromy homomorphism in the 
following way. Pull-back the flat principal $G$-bundle $E_G$ to the universal cover $\widetilde M$ of 
$M$. Since $\widetilde{M}$ is simply connected this pull-back is isomorphic to the
trivial principal $G$-bundle $\widetilde{M} \times G\,\longrightarrow\, \widetilde{M}$ equipped with the
trivial connection. Once again $E_G$ is isomorphic to a quotient of 
$\widetilde{M} \times G$ through the identification $(p,\,g)\,=\,(\gamma \cdot p,
\,\rho (\gamma) \cdot g)$, 
with $\gamma \in \pi_1(M)$ acting on $\widetilde{M}$ by deck-transformation and on $G$ through the 
monodromy morphism $\rho \,:\, \pi_1(M) \,\longrightarrow\, G$.
 
The pull-back of the subbundle $E_H$ is given by a smooth map from $\widetilde{M}$ to $X\,=\,G/H$ which 
is a local diffeomorphism (by the transversality condition): it is the developing map.
 
\section{Holomorphic Cartan geometries}\label{s2}

It is a very stringent condition for a compact complex manifold to admit a holomorphic Cartan geometry. In order to illustrate this phenomenon, let us remind a theorem of Wang \cite{Wa} which gives the classification
of those complex compact manifolds admitting a holomorphic Cartan geometry with model a complex Lie group $G$ (see the general Definition in Section {\ref{s2.1}).

\subsection{Parallelizable manifolds}

\begin{theorem}[Wang]
Let $M$ be a compact
complex manifold of dimension $m$ and $$\omega \,:\, TM \,\longrightarrow\, {\mathbb C}^m$$ a
holomorphic trivialization of the holomorphic tangent bundle. Then the universal cover of $M$
is biholomorphic to a complex Lie group $L$ and the pull-back of $\omega$ on $L$ coincides with
the Maurer-Cartan form of $L$. Consequently,
$d \omega + \frac{1}{2} \lbrack \omega,\, \omega \rbrack \,=\,0$.
 
Moreover, $M$ is K\"ahler if and only if $L$ is abelian (and $M$ is a complex torus).
\end{theorem}

\begin{proof}
Consider $X_{1}, X_{2}, \ldots, X_{m}$ global holomorphic $\omega$-constant vector fields on $M$ 
which span $TM$. Then, for all $1\,\leq\, i,\,j \,\leq\, m$, we have
$$\lbrack X_{i}, \,X_{j} \rbrack \,=\,f_{1}^{ij} X_{1} + f_{2}^{ij} X_{2} +\ldots+ f_{m}^{ij}X_{m}$$ 
with $f_{k}^{ij}$ being holomorphic functions on $M$. Since $M$ is compact, these functions have to be constant and, consequently, 
$X_{1},\, X_{2},\, \cdots, \,X_{m}$ generate a $m$-dimensional Lie algebra $\mathfrak L$. If
we consider on the target vector space of $\omega$ the Lie algebra structure of $\mathfrak L$, then
$\omega$ becomes a Lie algebra isomorphism.

By Lie's theorem there exists a unique connected simply connected complex Lie group $L$ 
corresponding to $\mathfrak L$. The holomorphic parallelization by $\omega$-constant holomorphic 
vector fields is locally isomorphic to the parallelization given by left translation-invariant 
vector fields on the Lie group $L$.

Since $M$ is compact, the $X_{i}$ are complete and they define a holomorphic transitive action 
of $L$ on $M$ (with discrete kernel). Hence $M$ is biholomorphic to a quotient of $L$ by a cocompact discrete subgroup 
$\Gamma$ in $L$.

The Lie-Cartan formula
$$d\omega (X_{i},\,X_{j})\,=\,X_i \cdot \omega (X_j)-X_j \cdot \omega(X_i) -
\omega (\lbrack X_{i},\,X_{j} \rbrack)\,=\,-\omega (\lbrack X_{i},\,X_{j} \rbrack)\,=\,
-\lbrack \omega(X_i), \,\omega (X_j) \rbrack $$
shows that $\omega$ satisfies the Maurer-Cartan equation of the Lie group $L$, which can also be put under the more formal expression $d \omega + \frac{1}{2} \lbrack \omega, \omega \rbrack=0$.

Assume now $M$ is K\"ahler. Then, any holomorphic form on $M$ has to be closed.
The Maurer-Cartan formula shows that the one-forms composing $\omega$ are all closed if and only if $L$ is abelian and thus $M$ is a complex torus.
\end{proof}

\subsection{Holomorphic Cartan geometry}\label{s2.1}

We recall now the definition of a Cartan geometry in the category of complex analytic manifolds. 
Let $G$ be a connected complex Lie group and $H\, \subset\, G$ a connected
complex Lie subgroup. The Lie algebras of $H$ and $G$ will be denoted by
$\mathfrak h$ and $\mathfrak g$ respectively.

The following definition generalizes the standard fibration $G \,\longrightarrow\, G/H$ seen as 
a $H$-(right) principal bundle and equipped with the left-invariant Maurer-Cartan form of $G$ (see also \cite{Sh}).

\begin{definition}\label{del}
A holomorphic Cartan geometry of type $(G,H)$ on a {complex} manifold $M$ is a principal
$H$--bundle $\pi_M\,: E_H\, \longrightarrow\, M$ endowed with 
 a $\mathfrak g$--valued holomorphic $1$--form $\omega$ on $E_H$ satisfying the following
conditions:
\begin{enumerate}
\item $\omega \, :\, TE_H\, \longrightarrow\, E_H\times{\mathfrak g}$ is an
isomorphism.

\item $\omega$ is $H$--equivariant with $H$ acting on $\mathfrak g$ via conjugation.

\item the restriction of $\omega$ to each fiber of $\pi_M$ coincides with the Maurer--Cartan form
associated to the action of $H$ on $E_H$.
\end{enumerate}

The Cartan geometry is called {\it flat} if its curvature $K(\omega)\,=\, d \omega + \frac{1}{2} 
\lbrack \omega, \omega \rbrack$ vanishes.
\end{definition}

Notice that $\omega$-constant holomorphic vector fields on $E_H$ form a holomorphic 
parallelization of the holomorphic tangent bundle of $E_H$. Consider the family of 
$\omega$-constant vector fields on $E_H$ with values in $\mathfrak h$. The form $\omega$
takes the Lie bracket of two such vector fields to the Lie product of the corresponding
elements of $\mathfrak h$. Moreover, $\omega$ is a Lie algebra isomorphism from the family of 
$\omega$-constant vector fields to $\mathfrak g$ if and only if the Cartan geometry defined by 
$\omega$ is flat.

In this case we have the following classical:

\begin{theorem} [Cartan]
The curvature $d \omega + \frac{1}{2} \lbrack \omega, \omega \rbrack$ vanishes if and only if
$(E_H, \,\omega )$ is locally isomorphic to the left-invariant Maurer-Cartan form on $G
\,\longrightarrow\, G/H$.
\end{theorem}

A flat Cartan geometry on $M$, gives a $(G,X=G/H)$-structure on $M$ in Ehresmann's sense.

Let
\begin{equation}\label{e2}
E_G\, :=\, E_H\times^H G \, \stackrel{f_G}{\longrightarrow}\, M
\end{equation}
be the holomorphic principal $G$--bundle on $M$ obtained by extending the
structure group of $E_H$ using the inclusion of $H$ in $G$. So, $E_G$ is the
quotient of $E_H\times G$ where two points $(c_1,\, g_1),\, (c_2,\, g_2)\, \in\,
E_H\times G$ are identified if there is an element $h\, \in\, H$ such that
$c_2\,=\, c_1h$ and $g_2\,=\, h^{-1}g_1$. The projection $f_G$ in \eqref{e2} is induced
by the map $E_H\times G\, \longrightarrow\, M$, $(c,\, g)\,\longmapsto\, \pi_M(c)$.
The action of $G$ on $E_G$ is induced by the
action of $G$ on $E_H\times G$ given by the right--translation action of $G$ on itself.
Let ${\rm ad}(E_H)\,=\, E_H\times^H {\mathfrak h}$ and
${\rm ad}(E_G)\,=\, E_G\times^G {\mathfrak g}$ be the adjoint vector bundles for
$E_H$ and $E_G$ respectively. We recall that ${\rm ad}(E_H)$ (respectively,
${\rm ad}(E_G)$) is the quotient of $E_H\times \mathfrak h$ (respectively, $E_G\times
\mathfrak g$) where two points $(z_1,\, v_1)$ and $(z_2,\, v_2)$ are identified if there
is an element $g\, \in\, H$ (respectively, $g\, \in\, G$) such that
$z_2\, =\, z_1g$ and $v_1$ is taken to $v_2$ by the automorphism of the Lie algebra
$\mathfrak h$ (respectively, $\mathfrak g$) given by automorphism of the Lie group
$H$ (respectively, $G)$ defined by $y\,\longmapsto\, g^{-1}yg$.
We have a short exact sequence of holomorphic vector bundles
on $X$
\begin{equation}\label{e3}
0\, \longrightarrow\, {\rm ad}(E_H)\, \stackrel{\iota_1}{\longrightarrow}\, {\rm ad}(E_G)
\, \longrightarrow\, {\rm ad}(E_G)/{\rm ad}(E_H)\, \longrightarrow\,0\, .
\end{equation}
The holomorphic tangent bundle of a complex manifold $Y$ will be denoted by $TY$. Let
$$
{\rm At}(E_H)\,=\, (TE_H)/H\, \longrightarrow\, M \ \text{ and } \ 
{\rm At}(E_G)\,=\, (TE_G)/G\, \longrightarrow\, M
$$
be the Atiyah bundles for $E_H$ and $E_G$ respectively; see \cite{At}. Let
\begin{equation}\label{e4}
0\, \longrightarrow\, {\rm ad}(E_H)\, \stackrel{\iota_2}{\longrightarrow}\, {\rm At}(E_H)
\, \stackrel{q_H}{\longrightarrow}\, TM\, \longrightarrow\,0
\end{equation}
and
\begin{equation}\label{e5}
0\, \longrightarrow\, {\rm ad}(E_G)\, \stackrel{\iota_0}{\longrightarrow}\, {\rm At}(E_G)
\, \stackrel{q_G}{\longrightarrow}\, TM\, \longrightarrow\,0
\end{equation}
be the Atiyah exact sequences for $E_H$ and $E_G$ respectively; see \cite{At}. The
projection $q_H$ (respectively, $q_G$) is induced by the differential of the map
$\pi_M$ (respectively, $f_G$).
A holomorphic connection on a holomorphic principal bundle
is defined to be a holomorphic splitting of the Atiyah exact sequence associated to the
principal bundle \cite{At}. Therefore, a holomorphic connection on $E_G$ is a
holomorphic homomorphism
$$
\psi\, :\, {\rm At}(E_G)\, \longrightarrow\, {\rm ad}(E_G)
$$
such that $\psi\circ \iota_0\,=\, \text{Id}_{{\rm ad}(E_G)}$.

It is straightforward to check that a holomorphic Cartan geometry on $X$ of
type $(G,\, H)$ is a pair $(E_H,\, \theta)$, where
$E_H$ is a holomorphic principal $H$--bundle on $M$ and
$$
\theta\, :\, {\rm At}(E_H)\, \longrightarrow\, {\rm ad}(E_G)
$$
is a holomorphic isomorphism of vector bundles such that
$\theta\circ \iota_2 \,=\,\iota_1$ (see \eqref{e4} and \eqref{e3} for
$\iota_2$ and $\iota_1$ respectively). Therefore, we have
the following commutative diagram
$$
\begin{matrix}
0 & \longrightarrow & {\rm ad}(E_H) & \stackrel{\iota_2}{\longrightarrow} & {\rm At}(E_H)
& \stackrel{q_H}{\longrightarrow} & TM & \longrightarrow & 0\\
&& \Vert && ~\Big\downarrow\theta && ~\Big\downarrow \phi\\
0 & \longrightarrow & {\rm ad}(E_H) & \stackrel{\iota_1}{\longrightarrow} & {\rm ad}(E_G)
& \longrightarrow & {\rm ad}(E_G)/{\rm ad}(E_H) & \longrightarrow & 0
\end{matrix}
$$
and the above homomorphism $\phi$ induced by
$\theta$ is evidently an isomorphism.

The following proposition is classical (see also \cite{Sh}):

\begin{proposition}\label{existence connection}
The $G$-principal bundle $E_G$ admits a holomorphic connection.
\end{proposition} 

\begin{proof} Let us consider on $E_H \times G$ the holomorphic one-form with values in $\mathfrak g$:

$$\widetilde{\omega}(c,g)= Ad(g^{-1})\pi_1^*(\omega) + \pi_2^* (\omega_G)$$

where $\pi_1$ and $\pi_2$ are the projections on the first and the second factor respectively and $\omega_G$ is the left-invariant Maurer-Cartan form on $G$.

One verify that $\widetilde{\omega}$ is invariant by the previous $H$-action and vanishes on the fibers of the fibration $E_H \times G \,\longrightarrow\, E_H \times^H G$.

This implies that $\widetilde{\omega}$ is basic (it is the pull-back of a form on $E_H \times^H G$, also denoted by $\widetilde{\omega})$.

The one-form $\widetilde{\omega}$ is a holomorphic connection on the principal $G$-bundle $E_G$.
\end{proof}

Consider $ f_G\,=\, E_G \,\longrightarrow\, M$ and a second projection map $h_G\,:\, E_G
\,\longrightarrow\, G/H$ 
which associates to each class $(c,\,g) \,\in\, E_G$ the element $g^{-1}H \,\in\, G/H$.

The differentials of the projections $f_G$ and $h_G$ map the horizontal space of the connection 
$\widetilde{\omega}$ at $(c,\,g) \,\in\, E_G$ isomorphically onto $T_{f_G(c)} M$ and onto $T_{g^{-1}H} (G/H)$ 
respectively. This defines an isomorphism between $T_{f_G(c)} M$ and $T_{g^{-1}H} (G/H)$. Hence a 
Cartan geometry provides a family of $1$-jets identifications of the manifold $X$ with the 
model space $G/H$.

Also, the connection $\widetilde{\omega}$ defines a way to lift differentiable curves in $M$ to 
$\widetilde{\omega}$-horizontal curves in $E_G$. Moreover this lift is unique once one specifies the 
starting point of the lifted curve.

When $\widetilde{\omega}$ is flat, its kernel is a foliation. In this case one can lift to $E_G$ not only curves, but also simply connected open sets in $M$ (their lift is tangent to the foliation given by the kernel of $\widetilde{\omega})$. Then their projection $h_G$ to the model $G/H$ is a local biholomorphism. The projections in $G/H$ of two different lifts (in two different leaves of the foliation) are related by the action of an element in $G$ acting on the model $G/H$. This procedure furnishes a $(G,X)$-structure on $M$ (compatible with the complex structure on $M$).

Moreover, if $M$ is simply connected one can lift to $E_G$ all of the manifold $M$. This defines a global section of the principal bundle $E_G$ which must be trivial: $E_G=M \times G$. In this trivialization our section
is $M \times \{e\}$ (with $e$ the identity element in $G$). The restriction of the projection
$h_G \,:
\, E_G \,\longrightarrow\, G/H$ to the section $M \times \{e \}$ defines a local biholomorphism from $M$ to the model $G/H$, which is
a developing map.

If $M$ is not simply connected, one can still pull-back the flat Cartan geometry to the universal cover $\widetilde M$ of $M$ and define in the same way the developing map which is a local biholomorphism from $\widetilde M$ to
the model $G/H$.

\subsection{Generalized holomorphic Cartan geometry}\label{s2.2}

This more general framework (similar to that of \cite{AM}) is given by the following:

\begin{definition} \label{general}
A generalized holomorphic Cartan geometry of type $(G,H)$ on a {complex} manifold $M$ is a principal $H$--bundle $\pi_M\,: E_H\, \longrightarrow\, M$ endowed with 
a $\mathfrak g$--valued holomorphic $1$--form $\omega$ on $E_H$ satisfying:
\begin{enumerate}

\item $\omega$ is $H$--equivariant with $H$ acting on $\mathfrak g$ via conjugation.

\item the restriction of $\omega$ to each fiber of $f$ coincides with the Maurer--Cartan form
associated to the action of $H$ on $E_H$.
\end{enumerate}
\end{definition}

Therefore, generalized holomorphic Cartan geometry of type $(G,H)$ is defined by a commutative diagram
\begin{equation}\label{e6a}
\begin{matrix}
0 & \longrightarrow & {\rm ad}(E_H) & \stackrel{\iota_2}{\longrightarrow} & {\rm At}(E_H)
& \stackrel{q_H}{\longrightarrow} & TM & \longrightarrow & 0\\
&& \Vert && ~\Big\downarrow\theta && ~\Big\downarrow \phi\\
0 & \longrightarrow & {\rm ad}(E_H) & \stackrel{\iota_1}{\longrightarrow} & {\rm ad}(E_G)
& \longrightarrow & {\rm ad}(E_G)/{\rm ad}(E_H) & \longrightarrow & 0
\end{matrix}
\end{equation}
of holomorphic vector bundles on $M$. We note that the homomorphism $\phi$ in \eqref{e6a}
is induced by $\theta$.

Suppose we are given $\theta$ as above.
We can embed $\text{ad}(E_H)$ in ${\rm At}(E_H)\oplus {\rm ad}(E_G)$ by sending any
$v$ to $(\iota_2(v),\, -\iota_1(v))$ (see \eqref{e4}, \eqref{e3}). The Atiyah bundle
${\rm At}(E_G)$ is the quotient $({\rm At}(E_H)\oplus {\rm ad}(E_G))/\text{ad}(E_H)$
for this embedding. The inclusion of $\text{ad}(E_G)$ in ${\rm At}(E_G)$ in
\eqref{e5} is given by the inclusion $\iota_1$ or $\iota_2$ of $\text{ad}(E_G)$ in
${\rm At}(E_H)\oplus {\rm ad}(E_G)$ (note that they give the same homomorphism to
the quotient bundle $({\rm At}(E_H)\oplus {\rm ad}(E_G))/\text{ad}(E_H)$). The homomorphism
$$
{\rm At}(E_H)\oplus {\rm ad}(E_G)\, \longrightarrow\, {\rm ad}(E_G),\, \ \
(v,\, w) \, \longmapsto\, \theta(v)+w
$$
produces a homomorphism
$$
\theta'\, :\, {\rm At}(E_G)\,=\,
({\rm At}(E_H)\oplus {\rm ad}(E_G))/\text{ad}(E_H)\, \longrightarrow\, {\rm ad}(E_G)
$$
which satisfies the condition that $\theta'\circ \iota_0\,=\, \text{Id}_{{\rm
ad}(E_G)}$, meaning $\theta'$ is a holomorphic splitting of \eqref{e5}. Therefore,
$\theta'$ is a holomorphic connection on the principal $G$--bundle $E_G$.

The Cartan geometry is called flat if its curvature $K(\omega)\,=\, d \omega + \frac{1}{2} \lbrack 
\omega, \omega \rbrack$ vanishes. Note that $K(\omega)$ is the curvature of the above connection
$\theta'$, where $\theta$ is the homomorphism corresponding to the Cartan geometry defined by $\omega$.

Notice that the proof of Proposition \ref{existence connection} still works for generalized holomorphic Cartan geometries and endows the principal $G$-bundle $E_G$ with a holomorphic connection. This connection is flat if and only if $K(\omega)=0$.

In this case we get a generalized $(G,X)$-structure on $M$ in the sense of the following definition:

\begin{definition}\label{def generalized Cartan}
A generalized $(G,X)$-structure on a complex manifold $M$ is given by an open cover $(U_i)_{i \in I}$ 
of $M$ with maps $\phi_i \,:\, U_i \,\longrightarrow\, X$ such that for each pair 
$(i,j)$ and for each nontrivial connected component of $U_i \cap U_j$ there exists an element 
$g_{ij} \in G$ such that $\phi_i\,=\,g_{ij} \circ \phi_j$ and $g_{ij}g_{jk}g_{ki}\,=\,1$.
\end{definition}

We note that the cocycle condition $g_{ij}g_{jk}g_{ki}\,=\,1$ is automatically satisfied if 
the element $g_{ij}$ such that $\phi_i=g_{ij} \circ \phi_j$ is unique. It may be clarified 
that in the previous definition $G/H$ and $M$ might have different dimensions.
 
The proof of Theorem \ref{Ehresmann} adapts to this situation and gives the following 
description:
 
\begin{theorem} \label{Ehresmann generalized} 
Let $X\,=\,G/H$ be a Klein geometry and $M$ a manifold. Then there is a canonical
one-to-one correspondence between the following structures.
\begin{enumerate}
\item[(1).] A generalized (G,X)-structure on $M$.

\item[(2).] A fiber bundle $E_X$ over $M$ with structure group $G$ and fiber $X$,
equipped with a flat connection and a global section.

\item[(3).] A principal $H$-bundle $E_H$ over $M$ and a flat
connection on the associated principal $G$-bundle $E_G\,=\,E_H \times^H G$.
\end{enumerate}
\end{theorem}

\begin{proof}
Let us first prove that (1) implies (2). Consider a cover of $M$ by open subsets $U_i$ and local charts $\phi_i \,:\, U_i 
\,\longrightarrow\, X$ such that on intersections $U_i \cap U_j$ there exist elements $g_{ij}\,\in\, G$ such that $\phi_i \circ 
\phi_j^{-1}\,=\,g_{ij}$ and $g_{ij}g_{jk}g_{ki}\,=\,1$. Then $g_{ij}$ is a locally constant cocycle with values in $G$ and hence it defines a 
flat $G$-bundle $E_G$ over $M$. Consider the associated flat bundle $E_X$, with fiber $X
\,=\,G/H$. Then the maps $\phi_i$ are local 
sections of $E_X$ which glue together in a global section of $E_X$. The flat connection of $E_X$ is locally determined by the 
foliation $U_i \times \{x\}$, for $x \,\in\, X$.

Remark that, since the fiber $X$ is the homogeneous space $G/H$, the existence of a global section of $E_X$ implies that the structural group  of $E_G$ reduces to $H$. Equivalently $E_G$ admits a $H$-principal subbundle $E_H$. Hence we proved that (1) also implies (3). The same proof shows also that (2) implies (3).

Let us finish the proof by showing that (3) implies (1). Consider a cover of $M$ by simply connected open sets $U_i$. Since $E_G$ is flat it must be trivial, isomorphic to
$U_i \times G$,  over all simply connected open sets $U_i$. In restriction to each $U_i$ the subbundle $E_H$ is defined by a map $\phi_i \,:\,U_i \,\longrightarrow\, X\,=\,G/H$. On $U_i \cap U_j$ the two maps $\phi_i$ and $\phi_j$ agree up to multiplication by the element $g_{ij} \in G$.
This achieves the proof.
\end{proof}

Moreover, one gets a developing map of a generalized $(G,X)$-structure on $M$. Indeed, if $M$ is simply connected the flat bundle $E_X$ in (2) of previous theorem is trivial 
and the global section is a map from $M$ to $G/H$.
 
In general, one can consider the pull-back of the generalized $(G,X)$-structure to the universal cover $\tilde M$ of $M$ and the previous global section is the developing map $dev$ of the generalized $(G,X)$-geometry. Hence for holomorphic generalized Cartan geometries, this is a holomorphic map $dev$ from the universal cover $\widetilde{M}$ to $X$. Notice that the rank of the differential of the developing map coincides with the rank of $\omega$.

If $\omega$ is an isomorphism on an open dense set of $E_H$ we call the corresponding 
generalized Cartan geometry: {\it branched} Cartan geometry. In this case the Klein geometry $G/H$ and $M$ must have the same dimension. For a
flat branched Cartan geometry, the maps $\phi_i$ in 
Definition \ref{def generalized Cartan} are branched bi-holomorphic maps (i.e.  their differential is an 
isomorphism at the generic point) and elements $g_{ij}$ are unique. In this case the global section in point (2) of Theorem 
\ref{Ehresmann generalized} is transverse to the flat connection at the generic point (i.e. 
over an open dense set in $M$). Equivalently, in point (3) of Theorem \ref{Ehresmann 
generalized} the $H$-subbundle $E_H$ of $E_G$ is transverse to the flat connection at the 
generic point.

We note that a branched holomorphic Cartan geometry is defined by a commutative diagram
\begin{equation}\label{e6}
\begin{matrix}
0 & \longrightarrow & {\rm ad}(E_H) & \stackrel{\iota_2}{\longrightarrow} & {\rm At}(E_H)
& \stackrel{q_H}{\longrightarrow} & TM & \longrightarrow & 0\\
&& \Vert && ~\Big\downarrow\theta && ~\Big\downarrow \phi\\
0 & \longrightarrow & {\rm ad}(E_H) & \stackrel{\iota_1}{\longrightarrow} & {\rm ad}(E_G)
& \longrightarrow & {\rm ad}(E_G)/{\rm ad}(E_H) & \longrightarrow & 0
\end{matrix}
\end{equation}
of holomorphic vector bundles on $M$, such that
$\theta$ is an isomorphism over a nonempty open subset of $M$;
the homomorphism $\phi$ in \eqref{e6} is induced by $\theta$.

\subsection{Branched  affine and projective structures}

A {\it holomorphic affine structure} (or equivalently {\it holomorphic affine connection}) on a 
complex manifold $M$ of dimension $m$ is a holomorphic Cartan geometry of type $({\mathbb C}^m\rtimes\text{GL}(m, {\mathbb C}) 
, \text{GL}(m, {\mathbb C}))$ (see also \cite{MM,Sh}).
A branched holomorphic Cartan geometry of type $({\mathbb C}^m\rtimes\text{GL}(m, {\mathbb C}), 
\text{GL}(m, {\mathbb C}))$ will be called a {\it branched holomorphic affine structure} or a 
{\it branched holomorphic affine connection}.

We also recall that a {\it holomorphic projective structure} (or a {\it holomorphic projective connection}) on a complex manifold $M$ of dimension $m$ is a 
holomorphic Cartan geometry of type $(\text{PGL}(m+1,{\mathbb C}), Q)$, where $Q\, \subset\, 
\text{PGL}(m+1,{\mathbb C})$ is the maximal parabolic subgroup that fixes a given point for the 
standard action of $\text{PGL}(m+1,{\mathbb C})$ on ${\mathbb C}P^m$ (the space of lines in 
${\mathbb C}^{m+1}$). In particular, there is a standard holomorphic projective structure on 
$\text{PGL}(m+1,{\mathbb C})/Q\,=\, {\mathbb C}P^m$ (see also \cite{Sh, OT}).

We will call a branched holomorphic Cartan 
geometry of type $(\text{PGL}(m+1,{\mathbb C}), Q)$ a {\it branched holomorphic projective structure} or a {\it branched holomorphic projective connection.}

\begin{proposition} \label {algebraic proj struct}
Every compact complex projective manifold $M$ admits a branched flat holomorphic
projective structure.
\end{proposition}

\begin{proof}
Let $M$ be a compact complex projective manifold of complex dimension $m$. Then there exists a finite
surjective algebraic, hence holomorphic, morphism
$$
\gamma\, :\, M\, \longrightarrow\, {\mathbb C}P^m\, . 
$$

Indeed, one proves that the smallest integer $N$ for which there exists a finite 
morphism $f$ from $M$ to ${\mathbb C}P^N$ is $m$. If $N \,>\,m$, then there exists $P 
\,\in\, {\mathbb C}P^N \setminus f(M)$; now consider the projection $\pi \,:\, {\mathbb 
C}P^N \setminus \{P \} \,\longrightarrow\, {\mathbb C}P^{N-1}$. The fibers of $\pi \circ f$ must be 
finite (otherwise $f(M)$ would contain a line through $P$, hence $P$). Since $\pi 
\circ f$ is a proper morphism with finite fibers, it must be finite.

Now we can pull back the standard holomorphic projective structure on ${\mathbb C}P^m$
using the above map $\gamma$ to get a branched holomorphic projective structure on $M$.
\end{proof}

\begin{proposition}\label{simply connected}\mbox{}
\begin{enumerate}
\item[(i)] Simply connected compact complex manifolds do not admit any branched flat
holomorphic affine structure.

\item[(ii)] Simply connected compact complex manifolds admitting a branched flat holomorphic
projective structure are Moishezon.
\end{enumerate}
\end{proposition}

\begin{proof} 
(i) If, by contradiction, a simply connected compact complex manifold $M$ admits a branched flat 
holomorphic affine structure, then the developing map $dev \,:\, M \,\longrightarrow\, 
{\mathbb C}^m$ is holomorphic and nonconstant, which is a contradiction.

(ii) If $M$ is a simply connected manifold of complex dimension $m$ admitting a branched flat
holomorphic projective structure, then its developing map is a holomorphic map $dev \,:\, M \,
\longrightarrow\, {\mathbb C}P^m$ which is a local biholomorphism away from a divisor $D$ in
$M$. Thus, the algebraic dimension of $M$ must be $m$; consequently, $M$ is Moishezon (e.g. meromorphic functions on $M$ separate points in general position) \cite{Mo}.
\end{proof}

Since any given compact K\"ahler manifold is Moishezon   if and only if it is projective \cite{Mo}, one gets
the following:

\begin{corollary}\label{corollary simply connected}
A non-projective simply connected compact K\"ahler manifold does
not admit any branched flat holomorphic
projective structure.
\end{corollary}

In particular, non-projective $K3$ surfaces do not admit any branched flat holomorphic projective 
structure.

Notice that the classification of compact complex surfaces bearing branched holomorphic affine connections and branched holomorphic projective connections is not known. In particular, we do not know if every compact complex surface bearing a branched holomorphic affine connection (respectively, a branched holomorphic projective connection)
also admits a branched flat holomorphic affine connection (respectively, a branched flat holomorphic projective connection). Recall that in the classical (unbranched) setting, the corresponding questions were settled and answered in a positive way in \cite{IKO,KO,KO2}.

Also we do not know the classification of compact complex surfaces bearing {\it branched holomorphic Riemannian metrics}.

 Recall that a branched holomorphic Riemannian metric on a complex manifold of dimension $m$ is a branched Cartan geometry with model $(G, H)$, where $H\,=\, \text{O}(m, {\mathbb C})$ and $G\,=\, {\mathbb 
C}^m\rtimes\text{O}(m, {\mathbb C})$.

A natural related question is to address the corresponding Cartan equivalence problem in the branched case: among the holomorphic sections of the bundle of complex quadratic forms  $S^2(T^*M)$ which are nondegenerate over an open dense set in $M$, characterize those who define a branched holomorphic Riemannian metric on (entire)  $M$. A more general question is to characterize holomorphic symmetric differentials (see the study in \cite{BO, BO1, BO2}) which are given by branched holomorphic Cartan geometries.

\section{Holomorphic Connections and Calabi-Yau manifolds}

Recall that K\"ahler Calabi--Yau manifolds are compact complex K\"ahler manifolds 
with the property that the first Chern class (with real coefficients) of the 
holomorphic tangent bundle vanishes. By Yau's theorem proving Calabi's conjecture, 
those manifolds admit K\"ahler metrics with vanishing Ricci curvature \cite{Ya}. 
Compact K\"ahler manifolds admitting a holomorphic affine connection have vanishing 
real Chern classes \cite{At}; it was proved in \cite{IKO} using Yau's result that 
they must admit finite unramified coverings which are complex tori.

The aim of this section is to give a simple proof of the following result of \cite{BD} in the case of the {\it projective} simply connected  Calabi-Yau manifolds and to deduce some consequences about generalized Cartan geometries on those manifolds (Corollary \ref{flatness}).

\begin{theorem} [\cite{BD}] \label{trivial} Let $M$ be a compact simply connected K\"ahler Calabi-Yau manifold and
$E$ a holomorphic vector bundle over $M$. If $E$ is equipped with a holomorphic connection
$\nabla$, then $E$ is trivial and the connection $\nabla$ coincides with the trivialization.
\end{theorem}

\begin{corollary} \label{principal trivial}
Let $P$ be a holomorphic $G$-principal bundle over a compact simply connected K\"ahler Calabi-Yau manifold $M$. Assume that the complex Lie group $G$ is semi-simple or simply connected.
Then, if $P$ admits a holomorphic connection, then $P$ is trivial.
\end{corollary}

Let us first show that Theorem \ref{trivial} implies the Corollary.

\begin{proof}
Let $\rho \,:\, G \,\longrightarrow\, {\rm GL}(N, {\mathbb C})$ be a linear representation of $G$ with discrete kernel. The corresponding Lie algebra representation $\rho'$ is an injective Lie algebra homomorphism
from $\mathfrak g$ to $\mathfrak{ gl}(N, {\mathbb C}). $ For $G$ simply connected those representations exist by Ado's theorem. For $G$ complex 
semi-simple those representations do also exist (see Theorem 3.2, chapter XVII in \cite{Ho}).

Consider the associated holomorphic vector bundle $P_{\rho}$ with fiber type ${\mathbb C}^N$. Then $P_{\rho}$ inherits a holomorphic connection $\nabla_{\rho}$ and, by Theorem \ref{trivial}, this connection
$\nabla_{\rho}$ must be flat. Since the curvature of $\nabla_{\rho}$ is the image of the curvature of $\nabla$ through $\rho'$ and $\rho'$ is injective, it follows that $\nabla$ is also flat. Since $M$ is simply-connected
the flat principal bundle $P$ has to be trivial.
\end{proof} 

In order to see that the hypothesis on $G$ is necessary, let us prove the following proposition inspired by a construction of non integrable holomorphic one-forms in \cite{Br}.

\begin{proposition}\label{propc}
There exists a nontrivial holomorphic principal elliptic bundle over a projective $K3$ surface which admits a non-flat
holomorphic connection, but does not admit any flat holomorphic connection.
\end{proposition}

\begin{proof}
Let us first remark that there exists projective $K3$-surfaces with holomorphic 2-forms $\Omega$ whose periods belong to a lattice in $\mathbb C$. Indeed, one can
consider $\mathbb{ C}^2$ and let $\Lambda$ be the standard lattice generated by
$(1,0), (\sqrt{-1},0), (0,1), (0,\sqrt{-1})$. Consider the involution $$i\,:\, \mathbb {C}^2 / \Lambda\, \longrightarrow\,
\mathbb {C}^2 / \Lambda$$ of the
torus $Y\,=\,\mathbb {C}^2 / \Lambda$ defined by $(x, \,y) \,\longmapsto\, (-x,\, -y)$. The fixed points of $ i$ are the 16 points of order 2 of $Y$.
Let $S$ be the blow-up of $Y/i $ at these 16 points. This smooth complex projective surface $S$
is a K3 surface (in particular, it is simply connected): a holomorphic 2-form $\Omega$ on it is given by the pull-back of 
the constant form $dz_1\wedge dz_2$ on $\mathbb{C}^2$.

The cohomology group $H^2(S,\, \mathbb{Z})$ is generated by the 16 copies of ${\mathbb C}P^1$ (the exceptional divisors for
the above mentioned blow-up) and by the 
generators of $H^2(Y, \,\mathbb{Z}).$ Now the evaluation of $\Omega$ on each exceptional divisor is zero because the form $\Omega$ is 
pulled back from $Y/i$ and hence vanishes on the exceptional divisors. On each generator of $H^2(Y,\, \mathbb{Z})$ the evaluation of 
$\Omega$ is one of $1/4, -1/4, \sqrt{-1}/4, -\sqrt{-1}/4.$ Therefore, the periods of $\Omega$ lie in a lattice in $\mathbb C$.

Now consider any projective K3 surface S with a nontrivial holomorphic 2-form $\Omega$ whose periods belong to a lattice $\Lambda$ in 
$\mathbb C$. The elliptic curve $\mathbb{C}/\Lambda$ will be denoted by $E$.

We can choose an open cover $\{U_i\}$ of $S$, in analytic topology, such that all $U_i$ and all connected
components of $U_i \cap U_j$ are contractible.
Then on each $U_i$ there exists a holomorphic one-form $\omega_i$ such that $\Omega_i\,=\,d 
\omega_i$. On intersections $U_i \cap U_j$ , we get $\omega_i-\omega_j\,=\,dF_{ij}$, with $F_{ij}$ a holomorphic function defined on $U_i 
\cap U_j$.

On intersections $U_i \cap U_j \cap U_k$, the sum $F_{ij} +F_{jk}+F_{ki}$ is a locally constant cocycle which
represents the class of $\Omega$ in $H^2(S,\, { \mathbb C})$. Hence we can choose the forms $\omega_{i}$ and
the associated functions $F_{ij}$ in such a way that 
$F_{ij}+F_{jk}+F_{ki }$ belongs to the lattice of periods $\Lambda$ for every triple $i,\, j,\, k$.

Consider then every holomorphic function $F_{ij}$ as taking values in the translations group of $E$ and form the holomorphic principal 
elliptic bundle $W_E\, \longrightarrow\, E$
with typical fiber $E$ associated to this one-cocycle. The local forms on $U_i \times E$ which are $\pi_1^*(\omega_i 
)+\pi_2^*dz $,   where $\pi_1$ and $\pi_2$ are the projections on the first and the second factor respectively  and $ z$ is a translation invariant coordinate on the elliptic curve, glue to a global
holomorphic one-form $\omega$ which is a holomorphic
connection on this principal elliptic bundle $W_E$. The differential $d\omega$ projects on $S$ to $\Omega$ and does 
not vanish. Hence, the holomorphic connection $\omega$ is not flat: its curvature is $\Omega$.

Moreover, the above holomorphic principal bundle $W_E$ with fiber type $E$ does not admit any holomorphic flat connection. 
Indeed, since the base $S$ is simply connected if there exists a flat connection on $W_E$, then the holomorphic 
principal bundle $W_E$ is trivial, biholomorphic to $E \times S\, \longrightarrow\, S$. In this case the total space of
the bundle $W_E$ must be 
K\"ahler and all holomorphic one-forms on $W_E$ would be closed: a contradiction since $d \omega \,\neq\, 0$.
\end{proof}

Let us go back to the proof of Theorem \ref{trivial}. 

Let $M$ be a compact K\"ahler manifold of complex dimension $m$ endowed
with a K\"ahler metric $\omega$.
Recall that the degree of a torsionfree coherent analytic sheaf $F$ on $M$ is defined to be
$$
\text{degree}(F)\,:=\, \int_M c_1(\det F)\wedge \omega^{m-1}\, \in\, {\mathbb R}\, ,
$$
where $\det F$ is the determinant line bundle for $F$ \cite[Ch.~V, \S~6]{Ko}.
This degree is well-defined. Indeed, any two
first Chern forms for $\det F$ differ by an exact $2$--form on $M$, and
$$
\int_M (d\alpha)\wedge \omega^{m-1}\,=\, -\int_M \alpha\wedge d\omega^{m-1}\,=\, 0\, .
$$
In fact, this shows that degree is a topological invariant. Define $$\mu(F)\,:=\,
\frac{\text{degree}(F)}{\text{rank}(F)}\, \in\, {\mathbb R}\, ,$$ which is called the
{\it slope} of $F$.
A torsionfree coherent analytic sheaf $F$ on $M$ is called \textit{stable} (respectively,
\textit{semi-stable}) if for every coherent analytic subsheaf $V\,\subset\, F$
with  $\text{rank}(V)\, \in\, [1\, ,\text{rank}(F)-1]$, the inequality
$\mu(V) \, <\, \mu(F)$ (respectively, $\mu(V) \, \leq\, \mu(F)$)
holds (see \cite[Ch.~V, \S~7]{Ko}).

First recall the following very useful lemma from \cite{Bi}:

\begin{lemma} [\cite{Bi}] Let $E$ be a holomorphic vector bundle over a compact  K\"ahler Calabi-Yau manifold. Assume that $E$ admits a holomorphic connection, then $E$ is semi-stable.
\end{lemma}

\begin{proof} We assume, by contradiction, that $E$ is not semi-stable. Then there exists a maximal semi-stable coherent subsheaf $W$. This $W$ is destabilizing : it has the property that for any
coherent subsheaf $F$ of $E$ such that $W \subset F$ we have the inequality of slopes $\mu (W) > \mu(F)$. This is equivalent to the property that for any subsheaf $F'$ of $E/W$ we have $\mu (W) > \mu(F')$.

Now we show that $W$ is $\nabla$-stable. We consider the second fundamental form $$S\,:\, W
\,\longrightarrow\, \Omega^1(M) \otimes E/W$$ which associates to any local section $s$ of $W$ the projection of 
$\nabla_{\cdot} s$ on $E/W$.

Assume, by contradiction, that the image $S(W)$ is nontrivial. Since $W$ is semi-stable and $S(W)$ is a quotient of $W$ we get $\mu (S(W)) \geq \mu(W)$. On the other hand, $\mu(S(W))$ is at most
 the slope of the maximal semi-stable subsheaf of $\Omega^1(M) \otimes E/W$ which is less than $\mu_1 + \mu_2$, with $\mu_1$ the slope of the maximal semi-stable subsheaf of $\Omega^1(M)$ and $\mu_2$ the slope of the maximal semi-stable subsheaf of $E/W$. But the holomorphic tangent bundle of a Calabi-Yau manifold is semi-stable, meaning that all quotients of $TM$ have nonnegative slope. By duality, this implies that any subsheaf of $\Omega^1(M)$ have nonpositive slope and, consequently, $\mu_1 \leq 0$. On the other hand, we have seen that $\mu_2 < \mu(W)$. 

It follows that $\mu(W) \,\leq\, \mu(S(W)) \,\leq\, \mu_1 + \mu_2 \,<\, \mu(W)$: a contradiction.

This implies $S(W)$ is trivial and, thus $W$ is $\nabla$-stable.

Then the following Lemma \ref{nabla} will finish the proof. Indeed, Lemma \ref{nabla} implies that $W$ is a 
subbundle endowed with a holomorphic connection. Then the degree of $W$ must be zero and we get 
$\mu(W)\,=\,\mu(E)\,=\,0$: a contradiction.
\end{proof} 

\begin{lemma}\label{nabla} Let $E$ be a holomorphic vector bundle over a complex manifold $M$ and
$W \,\subset\, E$ a coherent subsheaf. Assume that $E$ admits a holomorphic connection such that
$W$ is $\nabla$-stable. Then $W$ is a subbundle of $E$.
\end{lemma}

\begin{proof} Pick any $p \,\in\, M$. Since $W$ is a coherent subsheaf one can consider local sections $s_i$ in the neighborhood of $p$ such that $s_i(p)$ span the fiber $W_p$ at $p$. By Nakayama's lemma (see \cite{GH}, p. 680) the sections $s_i$ span $W$ in the neighborhood of $p$. We are left to prove that the sections $s_i$ are linearly independent over the ring of local holomorphic functions.

Assume, by contradiction, that there exists  a nontrivial relation $\sum_i f_i \cdot s_i\,=
\,0$, with $f_i$ local holomorphic functions.

It follows that $f_i(p)\,=\,0$, for all $i$. Consider the integer $n>0$ which is the minimum of vanishing orders of $f_i$ at $p$ (we can assume that this minimum is the vanishing order of $f_1$ at $p$).

We will construct a nontrivial relation for which the minimum of vanishing orders is strictly smaller. 

In local coordinates, one can find a directional derivative (in the direction of some local vector field $v$), such that the holomorphic function $d f_1(v)$, has vanishing order $< n$ at $p$. Then applying $\nabla$ to our relation we get 
$$\nabla_v (\sum_i f_i \cdot s_i)\,=\,\sum_i s_i df_i(v) + \sum_i f_i \nabla_v(s_i)\,=\, 0\, .$$
Since $W$ is $\nabla$-stable, for all $i$ one gets local holomorphic functions $g_{ij}$ such that 
$\nabla_v s_i\,=\, \sum_j g_{ij}s_j$.

We get then the relation $$\sum_i s_i(df_i(v) + f_i \sum_j g_{ij}s_j)\,=\,0$$ for which the minimum of the vanishing orders of the coefficients at $p$ is $<n$.

This is a contradiction. Consequently, $W$ is locally free. It is a subbundle of $E$.
\end{proof}

To finish the proof of Theorem \ref{trivial} we make use of a result of Simpson (\cite[p.~40, Corollary~3.10]{Si}) which states that:

{\it A holomorphic vector bundle $E$ on a complex projective manifold $M$ admits a flat 
holomorphic connection if $E$ is semi-stable with $\text{degree}(E)\,=\, 0\,=\, 
ch_2(E)\wedge \omega^{m-2}$.} Notice that in this statement $\omega$ is a rational K\"ahler class of $M$ and $ch_2$ is the second Chern class.

Consequently,
if the Calabi--Yau manifolds $M$ in Theorem \ref{trivial} is
{\it projective}, and the cohomology class of the K\"ahler form $\omega$ is rational,
then  $E$ admits a flat holomorphic connection,
because $E$ is semi-stable with vanishing Chern classes. Moreover since $M$ is simply connected
all flat bundles on $M$ must be trivial. Therefore, $E$ is the trivial vector bundle.
Since $H^0(M, \, \Omega^1_M)\,=\, 0$, the trivial holomorphic vector bundle has exactly
one holomorphic connection, namely the trivial connection. Hence
$\nabla$ is the trivial connection on $E$.

Theorem \ref{trivial} is now proved. This theorem has the following consequence.

\begin{corollary}\label{flatness}
Let $M$ be a compact simply-connected Calabi-Yau manifold endowed with a generalized holomorphic Cartan geometry $(E_H,\, \omega)$ modelled on $(G,H)$, with $G$ complex semi-simple Lie group (or $G$ simply connected complex Lie group). Then $(E_H, \omega)$ is flat. If $\omega$ is a branched Cartan geometry, $G/H$ must be compact and the branching set must be nontrivial.
\end{corollary}

\begin{proof}
Let $E_G$ be the associated $G$-principal constructed from $E_H$ by extension of the structure group. Now Corollary \ref{principal trivial} 
shows that the induced connection $\widetilde{\omega}$ on $E_G$ is flat. Consequently, the geometry $(E_H,\, \omega)$ is flat.

Assume that the Cartan geometry is branched. The associated developing map is a holomorphic map
$dev \,:
\, M \,\longrightarrow\, G/H$, which is a submersion away from the branching set. Since $M$ is compact,
the image of the developing map must be closed (and open) and it is all of the model $G/H$. That implies that
$G/H$ is compact.

Assume, by contradiction, that the branching set is empty. Then the developing map is a local biholomorphism, so
the developing map must be a covering map. It follows that $M$ is also a homogeneous space. Consider
global holomorphic vector fields
$X_1,\, X_2, \, \cdots ,\,X_m$ on $M$ which are linearly independent and
together span the holomorphic tangent bundle $TM$ at the generic point. Then $X_1 \bigwedge X_2 \bigwedge \cdots \bigwedge X_m$ is
a holomorphic section of $\bigwedge^mTM$. The vanishing divisor of this section is the first Chern class of $TM$ which is zero.
Consequently, the section $X_1 \bigwedge X_2 \bigwedge \cdots \bigwedge X_m$ does not vanish, and the vector fields
$(X_1,\, \cdots,\, X_m)$ trivialize the holomorphic tangent bundle $TM$. By Wang theorem, the complex manifold
$M$ is biholomorphic to a quotient of a complex connected Lie group $L$ by a lattice $\Gamma\, \subset\, L$. Now since $M$
is K\"ahler, and it is biholomorphic to $L/\Gamma$, it follows that $M$ must be a complex torus, in particular
the fundamental group of $M$ is infinite. This is a contradiction because
$M$ is simply connected.
\end{proof}

From Corollary \ref{flatness} and from Corollary \ref{corollary simply connected} we get the following:

\begin{corollary} \label{final}\mbox{}
Non-projective compact simply connected Calabi--Yau manifolds do not admit branched holomorphic projective structures.
\end{corollary}

Observe that Corollary \ref{principal trivial} and Corollary \ref{flatness} stand for all complex Lie groups $G$ which admit linear 
representations with discrete kernel.

\section*{Acknowledgements}

We thank the referee for helpful comments. IB is partially supported by a J. C. Bose Fellowship.


\end{document}